\documentclass[a4paper,12pt]{article}
\usepackage{amsfonts}
\usepackage{mathrsfs}
\usepackage{algorithm, algpseudocode}
\usepackage{amsmath,amssymb,amsthm,latexsym,amsfonts}
\usepackage{pstricks}
\usepackage{enumerate}

\usepackage{epsfig}
\usepackage{epstopdf}

\usepackage{graphicx}
\usepackage{color}
\usepackage{ifpdf}

\usepackage{caption}

\begin{document}

\newtheorem{lem}{Lemma}
\newtheorem{thm}{Theorem}
\newtheorem{cor}{Corollary}
\newtheorem{exa}{Example}
\newtheorem{con}{Conjecture}
\newtheorem{rem}{Remark}
\newtheorem{obs}{Observation}
\newtheorem{definition}{Definition}
\newtheorem{pro}{Proposition}
\theoremstyle{plain}
\newcommand{\D}{\displaystyle}
\newcommand{\DF}[2]{\D\frac{#1}{#2}}

\renewcommand{\figurename}{{\bf Fig}}
\captionsetup{labelfont=bf}

\title{\bf \Large Nordhaus-Gaddum-type theorem for
total proper connection number of
graphs\footnote{Supported by NSFC No.11371205 and 11531011, and PCSIRT.}}

\author{{\small Wenjing Li, Xueliang Li, Jingshu Zhang}\\
      {\small Center for Combinatorics and LPMC}\\
      {\small Nankai University, Tianjin 300071, China}\\
       {\small liwenjing610@mail.nankai.edu.cn; lxl@nankai.edu.cn;
       jszhang@mail.nankai.edu.cn}
       }
\date{}

\maketitle
\begin{abstract}
A graph is said to be \emph{total-colored} if all the edges and the vertices of the graph are colored. A path $P$ in a total-colored graph $G$ is called a \emph{total-proper path} if $(i)$ any two adjacent edges of $P$ are assigned distinct colors; $(ii)$ any two adjacent internal vertices of $P$ are assigned distinct colors; $(iii)$ any internal vertex of $P$ is assigned a distinct color from its incident edges of $P$. The total-colored graph $G$ is \emph{total-proper connected} if any two distinct vertices of $G$ are connected by a total-proper path. The \emph{total-proper connection number} of a connected graph $G$, denoted by $tpc(G)$, is the minimum number of colors that are required to make $G$ total-proper connected. In this paper, we first characterize the graphs $G$ on $n$ vertices with $tpc(G)=n-1$. Based on this, we obtain a Nordhaus-Gaddum-type result for total-proper connection number. We prove that if $G$ and $\overline{G}$ are connected complementary graphs on $n$ vertices, then $6\leq tpc(G)+tpc(\overline{G})\leq n+2$. Examples are given to show that the lower bound is sharp for $n\geq 4$. The upper bound is reached for $n\geq 5$ if and only if $G$ or $\overline{G}$ is the tree with maximum degree $n-2$.
\\[2mm]


\end{abstract}

\section{Introduction}

All graphs considered in this paper are simple, finite, and undirected. We follow the terminology and notation of Bondy and Murty in~\cite{Bondy} for those not
defined here.
If $G$ is a graph and $A\subseteq V(G)$,
then $G[A]$ denotes the subgraph of $G$
induced by the vertex set $A$, and $G-A$
the graph $G[V(G)\backslash\ A]$. If $A=\{v\}$, then we write $G-v$ for short.
An edge $xy$ is called a \emph{pendant edge} if one of its end vertices, say $x$, has
degree one, and $x$ is called a \emph{pendant vertex}.
For a vertex $v\in V(G)$, we use $N_G(v)$ to denote the neighborhood of $v$ in $G$ and use $d_G(v)$ to denote the degree of $v$ in $G$, sometimes we simply write $N(v)$ and $d(v)$ if $G$ is clear. For graphs $X$ and $G$, we write $X\cong G$ if $X$ is isomorphic to $G$. Throughout this paper, $\mathbb{N}$ denotes the set of all
positive integers.

Let $G$ be a nontrivial connected graph with an \emph{edge-coloring c}
$:E(G)\rightarrow \{1,2,\dots,t\}$, $t\in \mathbb{N}$, where adjacent edges may
be colored with the same color. If adjacent edges of $G$ receive different colors by $c$, then $c$ is a \emph{proper coloring}. 
The minimum number of colors required in a proper coloring of $G$ is referred as the \emph{chromatic index} of $G$ and denoted by $\chi'(G)$. Meanwhile, a path in $G$ is called \emph{a rainbow path} if no two edges
of the path are colored with the same color. The graph $G$ is called \emph{rainbow connected}
if for any two distinct vertices of $G$, there is a rainbow path connecting them.
For a connected graph $G$, the \emph{rainbow connection number} of $G$, denoted by $rc(G)$, is defined as the
minimum number of colors that are required to make $G$ rainbow connected.
These concepts
were first introduced by Chartrand et al. in~\cite{Char1} and have been well-studied since then.
For further details, we refer the reader to a book~\cite{Li}.

Motivated by rainbow connection coloring and proper coloring in graphs, Borozan et al.~\cite{Magnant} introduced the concept of proper-path coloring. Let $G$ be a nontrivial connected graph with an edge-coloring. A path in $G$ is called a \emph{proper path}
if no two adjacent edges of the path are colored with the same color. 
The $k$-\emph{proper connection number} of a connected graph $G$, denoted by $pc_k(G)$, 
is defined as the minimum number of colors that are required in an edge-coloring of $G$ such that any two distinct vertices of $G$ are connected by $k$ internally pairwise vertex-disjoint proper paths.
We write $pc(G)$ for short when $k=1$. For more details, we refer to a dynamic survey~\cite{Li1}.

Jiang et al.~\cite{Jiang} introduced the analogous concept of total-proper connection of graphs. Let $G$ be a nontrivial connected graph with a \emph{total-coloring c} $:E(G)\cup V(G)
\rightarrow \{1,2,\dots,t\}$, $t\in \mathbb{N}$. We use $c(u),c(uv)$ to denote the colors assigned to the vertex $u\in V(G)$ and the edge $uv\in E(G)$, respectively. A path $P$ is called a \emph{total-proper path} if $(i)$ any two adjacent edges of $P$ are assigned distinct colors; $(ii)$ any two adjacent internal vertices of $P$ are assigned distinct colors; $(iii)$ any internal vertex of $P$ is assigned a distinct color from its incident edges of $P$. A total-coloring $c$ is a \emph{total-proper coloring} of $G$ if every pair of distinct vertices $u,v$ of $G$ is connected by a total-proper path in $G$. A graph with a total-proper coloring is said to be \emph{total-proper connected}. If $k$ colors are used, then $c$ is referred as a \emph{total-proper $k$-coloring}. The \emph{total-proper connection number} of a connected graph $G$, denoted by $tpc(G)$, is the minimum number of colors that are required to make $G$ total-proper connected. For the total-proper connection number of graphs, the following observations are immediate.

\begin{pro}\label{pro1}
Let $G$ be a connected graph on $n$ vertices. Then

$(i)\ tpc(G)=1$ if and only if $G=K_n;$

$(ii)\ tpc(G)\geq 3$ if $G$ is noncomplete.
\end{pro}

A Nordhaus-Gaddum-type result is a (tight) lower or upper bound on the sum or product of the values of a parameter for a graph and its complement. The name ``Nordhaus-Gaddum-type'' is given because Nordhaus and Gaddum~\cite{Nordhaus} first established the following type of inequalities for chromatic number of graphs in 1956. They proved that if $G$ and $\overline{G}$ are complementary graphs on $n$ vertices whose chromatic number are $\chi(G)$ and $\chi(\overline{G})$, respectively, then $2\sqrt{n}\leq \chi(G)+\chi(\overline{G})\leq n+1$. Since then, many analogous inequalities of other graph parameters have been considered, such as diameter~\cite{Harary}, domination number~\cite{Harary2}, proper connection number~\cite{Huang}, and so on. In this paper, we consider analogous inequalities concerning total-proper connection number of graphs. We prove that if both $G$ and $\overline{G}$ are connected, then
\begin{equation*}
6\leq tpc(G)+tpc(\overline{G})\leq n+2.
\end{equation*}

The rest of this paper is organized as follows: In Section 2, we list some useful known results on total-proper connection number. In Section 3, we first characterize the graphs $G$ on $n$ vertices with $tpc(G)=n-1$. Based on this result, we give the upper bound and show that this bound is reached for $n\geq 5$ if and only if $G$ or $\overline{G}$ is the tree with maximum degree $n-2$. In the final section, we give the lower bound and show that it is sharp for $n\geq 4$.

\section{Preliminaries}

In this section, we list some preliminary results and definitions on the total-proper coloring which can be found in~\cite{Jiang}.

\begin{pro}\label{pro2}~\cite{Jiang}
If $G$ is a nontrivial connected graph and $H$ is a connected spanning subgraph of $G$, then $tpc(G)\leq tpc(H)$. In particular, $tpc(G)\leq tpc(T)$ for every spanning tree $T$ of $G$.
\end{pro}

\begin{pro}\label{pro3}~\cite{Jiang}
Let $G$ be a connected graph of order $n\geq 3$ that contains a bridge. If $b$ is the maximum number of bridges incident with a single vertex in $G$, then $tpc(G)\geq b+1$.
\end{pro}

In~\cite{Jiang}, the authors determined the total-proper connection numbers of trees and complete bipartite graphs.

\begin{thm}\label{thm1}~\cite{Jiang}
If $T$ is a tree of order $n\geq 3$, then $tpc(T)=\Delta(T)+1$.
\end{thm}

A \emph{Hamiltonian path} in a graph $G$ is a path containing every vertex of $G$ and a graph having a Hamiltonian path is a \emph{traceable graph}.

\begin{cor}\label{cor1}~\cite{Jiang}
If $G$ is a traceable graph that is not complete, then $tpc(G)=3$.
\end{cor}
\begin{thm}\label{thm2}~\cite{Jiang}
Let $G=K_{s,t}$ denote a complete bipartite graph with $s\geq t\geq 2$. Then $tpc(G)=3$.
\end{thm}

Given a total-colored path $P=v_1v_2\dots v_{s-1}v_s$ between any two vertices $v_1$ and $v_s$, we denote by $start_e(P)$ the color of the first edge in the path, i.e. $c(v_1v_2)$, and by $end_e(P)$ the last color, i.e. $c(v_{s-1}v_s)$. Moreover, let $start_v(P)$ the color of the first internal vertex in the path, i.e. $c(v_2)$, and by $end_v(P)$ the last color, i.e. $c(v_{s-1})$. If $P$ is just the edge $v_1v_s$, then $start_e(P)=end_e(P)=c(v_1v_s)$, $start_v(P)=c(v_s)$ and $end_v(P)=c(v_1)$.

\begin{definition}\label{def1}
Let $c$ be a total-coloring of a graph $G$ that makes $G$ total-proper connected. We say that $G$ has the \emph{strong property} if for any pair of vertices $u,v\in V(G)$, there exist two total-proper paths $P_1,P_2$ between them (not necessarily disjoint) such that $(1)\ c(u)\neq start_v(P_i)$ and $c(v)\neq end_v(P_i)$ for $i=1,2$, and $(2)$ both $\{c(u),start_e(P_1),start_e(P_2)\}$ and $\{c(v),end_e(P_1),end_e(P_2)\}$ are $3$-sets.
\end{definition}

The authors in~\cite{Jiang} studied the total-proper connection number of $2$-connected graphs and gave an upper bound.

\begin{thm}\label{thm3}~\cite{Jiang}
Let $G$ be a $2$-connected graph. Then $tpc(G)\leq 4$ and there exists a total-coloring of $G$ with $4$ colors such that $G$ has the strong property.
\end{thm}

From Definition~\ref{def1} and Theorem~\ref{thm3}, we get the following.

\begin{cor}\label{cor2}
Let $H=G\cup \{v\}$ such that $H$ is connected. If there is a total-proper $k$-coloring $c$ of $G$ such that $G$ has the strong property, then $tpc(H)\leq k$.
\end{cor}

We also study the total-proper connection number of $H$ when $G$ is a complete bipartite graph, and get the exact value of $tpc(H)$.

\begin{lem}\label{lem1}
Let $H=K_{s,t}\cup \{v\}$ such that $H$ is connected, where $s\geq t \geq 2$. Then $tpc(H)=3$.
Moreover, $tpc(H')=3$.
\end{lem}

\begin{figure}[h,t,b,p]
\centering
\scalebox{1.2}[1.2]{\includegraphics{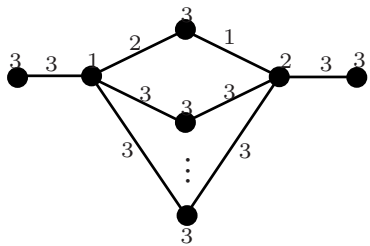}}\\
\caption{The graph $H'$}
\end{figure}

\begin{proof}
Let $U$ and $W$ be the two partite sets of $K_{s,t}$, where $U=\{u_1,\dots,u_s\}$ and $W=\{w_1,\dots,w_t\}$. Since $H$ and $H'$ are both noncomplete, we only need to prove $tpc(H)\leq 3$ and $tpc(H')\leq 3$, i.e., demonstrating a total-proper $3$-coloring of $H$ or $H'$. We divide our discussion according to the value of $t$.

{\bf Case~$1$.} $t=2$

If $v$ is adjacent to $W$, say $vw_1\in E(H)$, then set $c(w_1)=c(u_1w_2)=1$, and $c(w_2)=c(u_1w_1)=2$. Assign all the remaining vertices and edges with color $3$. Thus, there is a total-proper path $u_iw_1u_1w_2u_j$ connecting $u_i$ and $u_j$, where $2\leq i,j\leq s$. As for the rest of vertex pairs, we can always find a path contained in the path $vw_1u_1w_2u_i$ for some $2\leq i\leq s$. If there is another vertex $v'$ adjacent to $w_2$, based on the above coloring, set $c(v')=c(v'w_2)=3$, then we obtain a total-proper 3-coloring of $H'$, see Fig.1.

If $v$ is adjacent to $U$, say $vu_1\in E(H)$, then set $c(w_1)=c(u_2)=c(u_1w_2)=1$, and $c(w_2)=c(u_1w_1)=c(u_2w_1)=c(vu_1)=2$. Assign all the remaining vertices and edges with color $3$. Thus, there is a total-proper path, contained in the path $vu_1w_2u_2w_1$ or $vu_1w_2u_i$ for some $3\leq i\leq s$, connecting $v$ and any other vertex in $H$. And for vertex pairs in $U\cup W$, there is a total-proper path contained in the path $u_iw_2u_1w_1u_j$ for some $2\leq i<j\leq s$.

{\bf Case~$2$.} $t\geq 3$

If $s=t=3$, then $H$ is traceable so that $tpc(H)=3$.
If $s\geq 4$, we consider two subcases.

$1)$ Assume there is a $6$-cycle $C_6$ in $K_{s,t}$ such that $H-C_6$ is still connected. Without loss of generality, we suppose $C_6=u_1w_1u_2w_2u_3w_3$. We color $C_6$ with the colors $1,2,3$ by the sequence of vertices and edges on the cycle. That is, set $c(u_1)=c(w_2)=c(w_1u_2)=c(u_3w_3)=1$, $c(u_2)=c(w_3)=c(u_1w_1)=c(w_2u_3)=2$, and $c(w_1)=c(u_3)=c(u_2w_2)=c(w_3u_1)=3$.
Let $i,j\geq 4$ be two integers. Assign $u_i$ and $u_3w_j$ (if any) with color $1$, and assign $w_j$ and $w_1u_i$ with color $2$. The remaining vertices and edges are all colored $3$. Then we claim that this total-coloring makes $H$ total-proper connected. Any pair $(u_i,w_j) \in U\times W$ is connected by the edge $u_iw_j$. The total-proper path for the pairs from $U\times U$ is contained in the path $P=u_iw_1u_2w_2u_3w_3u_j$ for some $1\leq i,j \leq s$. And the total-proper path for the pairs from $W\times W$ is contained in the path $P=w_iu_1w_1u_2w_2u_3w_j$ for some $1\leq i,j \leq t$. Now consider the pairs of $\{v\}\times (U\cup W)$. By the assumption, we know that $vu_\ell\in E(H)$ or $vw_\ell\in E(H)$ for $\ell\geq 4$. Without loss of generality, suppose $\ell=4$. If $vu_4\in E(H)$, then for pairs $(v,u_i)$ ($1\leq i \leq s$) there is a total-proper path contained in the path $P=vu_4w_1u_2w_2u_3w_3u_j$ for some $1\leq j \leq s$, and for pairs $(v,w_i)$ ($1\leq i \leq t$) there is a total-proper path contained in the path $P=vu_4w_1u_2w_2u_3w_j$ for some $1\leq j \leq t$. The case when $vw_4\in E(H)$ is similar.

$2)$ Assume there is not such a $6$-cycle in subcase $1)$. As $s\geq 4$ we can deduce that $t=3$ and $v$ is only adjacent to $W$, say $vw_2\in E(H)$. Then we color $H$ as above. Then it is sufficient to check the pairs in $\{v\}\times (U\cup W)$. For pairs in
$\{v\}\times U$, there is a total-proper path $P=vw_2u_3w_3u_i$ for some $1\leq i\leq s$, and for pairs in $\{v\}\times W$, we can find a total-proper path contained in the path $P=vw_2u_3w_3u_1w_1$.

The proof is complete.
\end{proof}

\section{Upper bound on $tpc(G)+tpc(\overline{G})$}

At the beginning of this section, we give total-proper connection numbers of four unicyclic graphs, which are useful to characterize the graphs on $n$ vertices that have total-proper connection number $n-2$.

\begin{lem}\label{lem2}
Let $H_1,H_2,H_3$ and $H_4$ be the graphs on $n\geq 5$ vertices shown in the Fig.~$2$, respectively. Then $tpc(H_1)=n-2$; $tpc(H_2)=n-2$ if $n=5$, $tpc(H_2)=n-3$ if $n\geq 6$; and for $i=3,4$, $tpc(H_i)=n-2$ if $n=5$ or $6$, $tpc(H_i)=n-3$ if $n\geq 7$.
\end{lem}

\begin{figure}[h,t,b,p]
\centering
\scalebox{1.2}[1.2]{\includegraphics{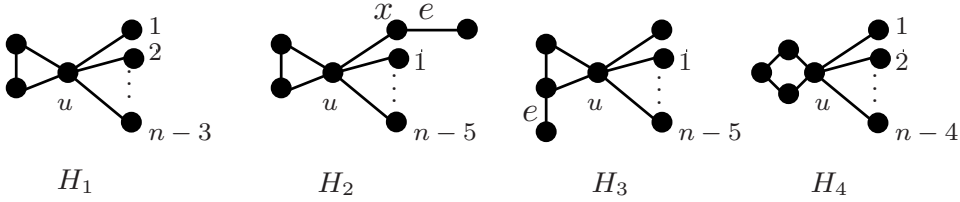}}\\
\caption{The graphs $H_1$, $H_2$, $H_3$ and $H_4$.}
\end{figure}

\begin{proof} By Proposition~\ref{pro3}, we get $tpc(H_1)\geq n-2$ and $tpc(H_i)\geq n-3$ for $i\in \{2,3,4\}$.

For $i=1,2,3$, let $uvw$ be the triangle in $H_i$ and let $e_1$, $e_2$,\dots, and $e_{n-3}$ denote the bridges in $H_i$. Assume that $e=e_{n-3}$ in the graphs $H_2$ and $H_3$, and the edge $e$ is incident with the vertex $x$ and adjacent to the bridge $e_1$ in $H_2$, and $e$ is incident with the vertex $v$ in $H_3$. We first consider the graph $H_1$ and demonstrate a total-coloring of it with $n-2$ colors. Let $c(u)=c(vw)=1$, $c(e_j)=j+1$ for $1\leq j\leq n-3$, $c(uv)=c(w)=2$ and $c(v)=c(wu)=3$. The remaining vertices are all colored~$1$. It is easy to check this total-coloring makes $H_1$ total-proper connected. Hence, we have $tpc(H_1)=n-2$ when $n\geq 5$.

We should point out that for $i=2,3,4$, the graph $H_i$ is traceable when $n=5$, hence $tpc(H_i)=3$ by Corollary~\ref{cor1}. So we assume $n\geq 6$. Consider the graph $H_2$. Color as $H_1$ only with the exception that $c(e_{n-3})=1$ and $c(x)=3$. It is easy to check that under this total-coloring, $H_2$ is total-proper connected. Hence, we have $tpc(H_2)=n-2$ when $n=5$ and $tpc(H_2)=n-3$ when $n\geq 6$.

Consider the graph $H_3$. When $n=6$, we claim that $tpc(H_3)=4$. From Proposition~\ref{pro2}, we get that $tpc(H_3)\leq 4$. If we use 3 colors to total-color $H_3$, 
no matter how we color it, there always exist two pendant vertices not being connected by a total-proper path. When $n\geq 7$, it can be easily checked that the total-coloring of $H_2$, 
only with the exception that $c(e)=4$, makes $H_3$ total-proper connected. Hence, we have $tpc(H_3)=n-2$ when $n=5,6$ and $tpc(H_3)=n-3$ when $n\geq 7$.

Now we consider the graph $H_4$. We use $e_1$, $e_2$,\dots, and $e_{n-4}$ to denote the bridges incident with $u$, respectively, 
and use $uvwx$ to denote the quadrangle in $H_4$. 
First, we consider the case $n\geq 7$. We demonstrate a total-coloring of $H_4$ with $n-3$ colors. Let $c(e_j)=j$ for $1\leq j\leq n-4$, $c(u)=n-3$, $c(v)=c(x)=2$, 
$c(vw)=c(xu)=3$ and $c(w)=4$. 
The remaining edges and vertices are all colored~$1$. It is easy to check that under this total-coloring, $H_4$ is total-proper connected. 
When $n=6$, we claim that $tpc(H_4)=4$. From Proposition~\ref{pro2}, we get that $tpc(H_4)\leq 4$.
If we use 3 colors to total-color $H_4$, no matter how we color it, there always exists a vertex pair not being connected by a total-proper path. 
Hence, we have $tpc(H_4)=n-2$ when $n=5,6$ and $tpc(H_4)=n-3$ when $n\geq 7$.
\end{proof}

We use $C_n$ and $S_n$ to denote the cycle and the star on $n$ vertices, respectively, and use $T(a,b)$ to denote the double star that is obtained by adding an edge between the center vertices of $S_a$ and $S_b$. Given a cycle $C_r=v_1v_2\dots v_r$, let $C_r(T_1,T_2,\dots,T_r)$ be the graph obtained from $C_r$ and rooted trees $T_i$ by identifying the root, say $r_i$, of $T_i$ with $v_i$ on $C_r$, $i=1,2,\dots, r$. We assume that $|T_i|=n_i, n_i\geq 1, i=1,2,\dots,r$. Then $|C_r(T_1,T_2,\dots,T_r)|=\sum_{i=1}^r|T_i|$. In particular, if $|T_i|=1$ for each $i\in\{1,2,\dots, r\}$, the graph $C_r(T_1,T_2,\dots,T_r)$ is just the cycle $C_r$. For a nontrivial graph $G$ such that $G+uv\cong G+xy$ for every two pairs $(u,v)$, $(x,y)$ of nonadjacent vertices of $G$, we use $G+e$ to denote the graph obtained from $G$ by joining two nonadjacent vertices of $G$.

\begin{thm}\label{thm4}
Let $G$ be a connected graph of order $n\geq 3$. Then $tpc(G)=n-1$ if and only if $G\in \{T(2,n-2), C_4, C_4+e,S_4+e\}$.
\end{thm}

\begin{proof}By Theorem~\ref{thm1} and Corollary~\ref{cor1}, we can easily check that $tpc(G)=n-1$ if $G$ is one of the above four graphs. 
So we concentrate on the verification of the converse of the theorem. Suppose that $tpc(G)=n-1$. Then $G$ cannot be complete, so $tpc(G)\geq 3$. 
If $G$ is a tree, then by Theorem~\ref{thm1}, we have $\Delta(G)=n-2$, thus $G\cong T(2, n-2)$. Now, we consider the case that $G$ contains cycles. 
Pick a longest cycle $C_k=v_1v_2...v_k$ of $G$, where $k\geq 3$. If $k=n$, then $3=tpc(C_k)=tpc(G)=n-1$.
So $n=4$. Thus $G\cong C_4$ or $C_4+e$. If $k<n$, consider a unicyclic spanning subgraph $H$ of $G$ containing the cycle $C_k$. 
Then $H$ can be written as $C_k(T_1,T_2,...,T_k)$. Set $r=max\{\Delta(T_i):1\leq i\leq k\}$ and let $T_\ell$ be a tree with $\Delta(T_\ell)=r$. 
Notice that $\Delta(T_\ell)\leq |T_\ell|-1\leq n-k$, so $r\leq n-k$. Then delete an edge $e$ of $H$, which is incident with $v_\ell$ in $C_k$, and denote 
the obtained graph as $H'$, so $H'$ is a spanning tree of $G$ and $\Delta (H')\leq n-k+1$, 
and the equality holds if and only if there is only one 
nontrivial subtree $T_\ell=S_{n-k+1}$ in $H$ whose 
center is $v_\ell$ or there are exactly two pendant edges 
attaching to $C_k$. Thus $n-1=tpc(G)\leq tpc(H')=\Delta(H')+1\leq n-k+2$,
therefore we have $k\leq 3$. So $k=3$ and all the equalities must hold. 
Hence, there is only one nontrivial subtree in $H$ and $\Delta(H)=n-1$ or $H$ is traceable on~$5$ vertices, 
the latter contradicting the condition $tpc(G)=n-1$. 
So we can identify $H$ as $S_n+e$, and when $n\geq 5$, the graph $H$ is just the graph $H_1$ in Fig.~$2$.
By Lemma~\ref{lem3} and Proposition~\ref{pro2}, we have $tpc(G)\leq tpc(H_1)=n-2$, a contradiction. 
So $n=4$ and $G\cong S_4+e$ since $C_3$ is a longest cycle of $G$.
\end{proof}

We know that if $G$ and $\overline{G}$ are connected complementary graphs on $n$ vertices, then $n$ is at least~$4$, and $\Delta(G)\leq n-2$. Therefore, we get that $tpc(G)\leq n-1$. Similarly, we have $tpc(\overline{G})\leq n-1$. Hence, we obtain that $tpc(G)+tpc(\overline{G})\leq 2(n-1)$. For $n=4$, it is obvious that $tpc(G)+tpc(\overline{G})=6$ if both $G$ and $\overline{G}$ are connected. In the rest of this section, we always assume that all graphs have at least~$5$ vertices, and both $G$ and $\overline{G}$ are connected.

\begin{lem}\label{lem3}
Let $G$ be a graph on~$5$ vertices. If both $G$ and $\overline{G}$ are connected, then we have
\begin{equation*}
tpc(G)+tpc(\overline{G})=\left\{\begin{array}{ll}
7  & \text{if $G\cong T(2,3)$ or $\overline{G}\cong T(2,3)$;}\\
6  & \text{otherwise.}
\end{array}\right.
\end{equation*}
\end{lem}

\begin{proof} If $G\cong T(2,3)$ or $\overline{G}\cong T(2,3)$, then from Theorem~\ref{thm4}, we can easily get that $tpc(G)+tpc(\overline{G})=7$. Otherwise, we have $tpc(G)\leq n-2=3$ and $tpc(\overline{G})\leq n-2=3$. Combining with Proposition~\ref{pro1}, we get $tpc(G)+tpc(\overline{G})=3+3=6$ if $G\ncong T(2,3)$ and $\overline{G}\ncong T(2,3)$.
\end{proof}

\begin{thm}\label{thm5}
Let $G$ be a graph of order~$n\geq 5$. If both $G$ and $\overline{G}$ are connected, then we have $tpc(G)+tpc(\overline{G})\leq n+2$, and the equality holds if and only if $G\cong T(2,n-2)$ or $\overline{G}\cong T(2,n-2)$.
\end{thm}

\begin{proof}
It follows from Lemma~\ref{lem3} that the result holds for $n=5$. So we assume that $n\geq 6$.
If $G\cong T(2,n-2)$, then $\overline{G}$ contains a spanning subgraph $H$ that is obtained by attaching a pendant edge to the complete bipartite graph $K_{2,n-3}$. So we have $tpc(G)=3$ by Lemma~\ref{thm1}. Combining with Theorem~\ref{thm4}, the result is clear. Similarly, we get that $tpc(G)+tpc(\overline{G})=n+2$ if $\overline{G}\cong T(2,n-2)$. In the following, we prove that
$tpc(G)+tpc(\overline{G})< n+2$ when $G\ncong T(2,n-2)$ and $\overline{G}\ncong T(2,n-2)$. Under this assumption, we have $3\leq tpc(G)\leq n-2$ and $3\leq tpc(\overline{G})\leq n-2$ by Proposition~\ref{pro1} and Theorem~\ref{thm4}.

We first consider the case that both $G$ and $\overline{G}$ are~$2$-connected. When $n=6$, we claim that $tpc(G)=3$. Suppose that the circumference of $G$ is $k$. If $k=6$, then $tpc(G)\leq tpc(C_6)=3$. If $k=4$, then $G$ contains a spanning $K_{2,4}$, contradicting the fact that $\overline{G}$ is connected.
Next, we assume that $G$ contains a~$5$-cycle $C=v_1v_2v_3v_4v_5$. Then $G$ is traceable, so $tpc(G)=3$ by Corollary~\ref{cor1}. Thus, we have $tpc(G)+tpc(\overline{G})\leq 3+n-2<n+2$.
For $n\geq 7$, we have $tpc(G)\leq 4$ and $tpc(\overline{G})\leq 4$ by Theorem~\ref{thm3}.
Hence, we get $tpc(G)+tpc(\overline{G})\leq 4+4<n+2$.

Now, we consider the case that at least one of $G$ and $\overline{G}$ has cut vertices. Without loss of generality, we suppose that $G$ has cut vertices. Let $u$ be a cut vertex of $G$, let $G_1,G_2,\dots,G_k$ be the components of $G-u$, and let $n_i$ be the number of vertices in $G_i$ for $1\leq i\leq k$ with $n_1\leq \dots\leq n_k$. We consider the following two cases.

{\bf Case 1.} There exists a cut vertex $u$ of $G$ such that $n-1-n_k\geq 2$. Since $\Delta(G)\leq n-2$, we have $n_k\geq 2$. We know that $\overline{G}-u$ contains a spanning complete bipartite graph $K_{n-1-n_k,n_k}$. Hence, it follows from Lemma~\ref{lem1} that $tpc(\overline{G})=3$. Combining with the fact that $tpc(G)\leq n-2$, we get that $tpc(G)+tpc(\overline{G})< n+2$.

{\bf Case 2.} Every cut vertex $u$ of $G$ satisfies that $n-1-n_k=1$.

First, we suppose that $G$ has at least two cut vertices, say $u_1$ and $u_2$. Let $u_1v_1$ and $u_2v_2$ 
 be two pendant edges of $G$. Obviously, the edges $u_1v_1$ and $u_2v_2$ are disjoint. 
 So $u_1v_2,u_2v_1\in E(\overline{G})$, and $\overline{G}-\{u_1,u_2\}$ contains a spanning complete bipartite graph $K_{2,n-4}$ with two partitions 
 $U=\{v_1,v_2\}$ and $W=V(G)\backslash \{u_1,v_1,u_2,v_2\}$. By Lemma~\ref{lem1}, 
 we have that $tpc(\overline{G})=3$. 
 Together with the fact that $tpc(G)\leq n-2$, we get that $tpc(G)+tpc(\overline{G})< n+2$.

Now, we consider the subcase that $G$ has only one cut vertex $u$ and let $uv$ be the pendant edge of $G$. Then $G-v$ is~$2$-connected. 
By Theorem~\ref{thm3} and Corollary~\ref{cor2}, we have $tpc(G)\leq 4$, thus $tpc(G)+tpc(\overline{G})\leq n+2$. 
Now, we prove that the equality cannot hold. Note that $d_{\overline{G}}(v)=n-2$. Let $N_{\overline{G}}(v)=\{w_1,w_2,\dots,w_{n-2}\}$. 
Since $\Delta(G)\leq n-2$, there exists a vertex $w_i$ $(1\leq i\leq n-2)$ not adjacent to $u$ in $G$, say $uw_1\notin E(G)$. 
Then $uw_1\in E(\overline{G})$. If there is a vertex $w_j$ $(2\leq j\leq n-2)$
adjacent to $w_1$ in $\overline{G}$, then $\overline{G}$ contains $H_3$ in Fig.~$2$ as a spanning subgraph,
so $tpc(\overline{G})\leq max\{4,n-3\}$. If there is a vertex $w_j$ $(2\leq j\leq n-2)$
adjacent to $u$ in $\overline{G}$, then $\overline{G}$ contains $H_4$ in Fig.~$2$ as a spanning subgraph, so $tpc(\overline{G})\leq max\{4,n-3\}$. 
If there are two vertices $w_j,w_k(2\leq j\neq k\leq n-2)$ 
adjacent in $\overline{G}$, then $\overline{G}$ contains $H_2$ in Fig.~$2$ as a 
spanning subgraph, so $tpc(\overline{G})\leq n-3$. We conclude that $tpc(\overline{G})\leq max\{4,n-3\}$ if
$G-v$ is~$2$-connected. For $n\geq 7$, 
we get the result $tpc(G)+tpc(\overline{G})\leq n+1<n+2$. 
For $n=6$, since $G-v$ is a~$2$-connected graph on~$5$ vertices, 
$G-v$ contains a spanning~$5$-cycle or a spanning $K_{2,3}$, 
implying that $tpc(G)=3$ by Corollary~\ref{cor1} and 
Lemma~\ref{lem1}. Thus, we have $tpc(G)+tpc(\overline{G})\leq 3+4=7<8$.
\end{proof}

\section{Lower bound on $tpc(G)+tpc(\overline{G})$}

As we have noted that $tpc(G)=1$ if and only if $G$ is a complete graph. In this case, the graph $\overline{G}$ is not connected. So, if $G$ and $\overline{G}$ are both connected, then $tpc(G)\geq 3$. Similarly, we have $tpc(\overline{G})\geq 3$. Hence, we obtain that $tpc(G)+tpc(\overline{G})\geq 6$.

\begin{thm}\label{thm6}
Let $G$ be a graph of order $n\geq 4$. If both $G$ and $\overline{G}$ are connected, then we have $tpc(G)+tpc(\overline{G})\geq 6$, and the lower bound is sharp.
\end{thm}

\begin{proof}
We only need to prove that there are graphs $G$ and $\overline{G}$ on $n\geq 4$ vertices such that $tpc(G)=tpc(\overline{G})=3$.

Let $G$ be the graph with vertex set $\{v\}\cup U\cup W$, where $U=\{u_1,\dots,u_{\lfloor\frac{n-1}{2}\rfloor}\}$ and $W=\{w_1,\dots,w_{\lceil\frac{n-1}{2}\rceil}\}$, such that $N(v)=U$ and $U$ is an independent set and $G[W]$ is a clique, and for each vertex $u_i$, $u_i$ is adjacent to $w_i,w_{i+1},\dots,w_{i+\lfloor\frac{n-3}{4}\rfloor}$ where the subscripts are taken modulo $\lceil\frac{n-1}{2}\rceil$. Obviously, the graphs $G$ and $\overline{G}$ are both traceable. It follows from Corollary~\ref{cor1} that $tpc(G)=tpc(\overline{G})=3$.
\end{proof}


\end{document}